\title{Limits of compact decorated graphs}
\author{{\sc L\'aszl\'o Lov\'asz}\footnote{Research supported by
OTKA grant No. 77780 and ERC Advanced research grant No. 227701},
E\"otv\"os Lor\'and University, Budapest\\and\\
{\sc Bal\'azs Szegedy}, University of Toronto, Toronto}

\date{Oct 2010}

\documentclass[11pt]{article}
\usepackage{amssymb,amsmath,graphicx,ntheorem}
\usepackage{hyperref}

\def\CC{{\cal C}}\def\FF{{\cal F}}
\def\GG{{\cal G}}\def\KK{{\cal K}}
\def\PP{{\cal P}}
\def\SS{{\cal S}}\def\WW{{\cal W}}

\def\R{{\mathbb R}}
\def\Z{{\mathbb Z}}
\def\N{{\mathbb N}}

\def\C{{\mathbb C}}

\def\eps{\varepsilon}

\def\hom{{\rm hom}}

\newtheorem{theorem}{Theorem}[section]

\newtheorem{lemma}[theorem]{Lemma}

\theorembodyfont{\rmfamily}

\newtheorem{example}[theorem]{Example}

\newtheorem{definition}[theorem]{Definition}

\long\def\killtext#1{}

\newenvironment{proof}{\medskip\noindent{\bf Proof. }}{\hfill$\square$\medskip}
\newenvironment{proof*}[1]{\medskip\noindent{\bf Proof of #1. }}{\hfill$\square$\medskip}

\begin{document}

\maketitle

\tableofcontents

\begin{abstract}
Following a general program of studying limits of discrete
structures, and motivated by the theory of limit objects of converge
sequences of dense simple graphs, we study the limit of graph
sequences such that every edge is labeled by an element of a compact
second-countable Hausdorff space $\KK$. The "local structure" of
these objects can be explored by a sampling process, which is shown
to be equivalent to knowing homomorphism numbers from graphs whose
edges are decorated by continuous functions on $\KK$. The model
includes multigraphs with bounded edge multiplicities, graphs whose
edges are weighted with real numbers from a finite interval,
edge-colored graphs, and other models. In all these cases, a limit
object can be defined in terms of 2-variable functions whose values
are probability distributions on $\KK$.
\end{abstract}

\section{Introduction}

This paper fits into a general program 
in the frame of which limits of discrete structures are studied. The
main assumption is that the objects have a "local structure" which
can be explored by certain sampling processes.

A typical example is subsets of integer intervals. If
$H\subseteq\{1,2,3,\dots,n\}$ is such an object then we can sample
$k$ consecutive elements from $\{1,2,\dots,n\}$ uniformly at random
and take the intersection of $H$ with it. The corresponding
convergence notion leads to limit objects that are shift invariant
measures on the compact space $\{0,1\}^\mathbb{Z}$; these measures
are important in ergodic theory.

Another example, more relevant for us, is the set of finite simple
graphs. For every natural number $k$ there is a sampling process in
which we pick $k$ random nodes and look at the subgraph induced by
them. A sequence $G_1,G_2,\dots$ of simple graphs with
$|V(G_n)|\to\infty$ is called {\it convergent} if the distribution of
this random induced subgraph is convergent for every $k$. To every
convergent sequence of simple graphs one can assign a limit object in
the form of a 2-variable real function \cite{LSz1} (see below).

Our main goal is to generalized these results to limits of multigraph
sequences, moments indexed by multigraphs, and beyond. We study the
limit of graph sequences such that every edge is labeled by an
element of a fixed second-countable compact Hausdorff space $\KK$.
This includes weighted graphs with bounded edge-weights, or
multigraphs with bounded edge multiplicities.

We define convergence of compact decorated graph sequences, and
construct limit objects for such sequences. We introduce a notion of
homomorphism numbers into such graphs, and show that convergence can
be characterized in terms of them.

\section{Overview of results}

\subsection{Decorated graphs}

In this section we develop a formalism which provides a unified
treatment for many similar problems.

Let $S$ be an arbitrary set. For $n\in\Z_+$, a symmetric map
$G:~[n]\times [n]\to S$ (i.e., a map satisfying $G(x,y)=G(y,x)$) is
called an {\it $S$-decorated graph}. We can also think of this as an
assignment of an element of $S$ to every edge of $\tilde K_n$, the
complete graph on the node set $[n]=\{1,2,\dots,n\}$ with a loop edge
on every node; alternatively, $G$ is an $n\times n$ symmetric matrix
with entries from $S$. We denote by $\GG_n(S)$ the set of all
$S$-decorated graphs with $n$ nodes, and set $\GG(S)=\cup_{n\ge 0}
\GG_n(S)$. Often $S$ will have a special element $0$, and an edge
decorated with $0$ will be considered as missing. Often our graphs
will have no loops, i.e., $G(x,x)=0$ for all $x\in[n]$.

If $S$ is finite, then so is $\GG_n(S)$. If $S$ is a topological
space then, for every $n$, $\GG_n(S)$ is a topological spaces with
the product topology. If, in addition, $S$ is compact, then so is
$\GG_n(S)$.

\subsection{Sampling}

For every natural number $k$ and $G\in\GG(S)$ there is a {\it
sampling process} $\mathbb{G}(G,k)$ which is a random variable whose
values are in $\GG_k(S)$, and is defined as follows. We pick a random
ordered set of $k$ nodes $\{v_1,v_2,\dots,v_k\}$ uniformly from $G$
and then we create a graph $F=\mathbb{G}(G,k)\in\GG_k(S)$ on the node
set $\{1,2,\dots,k\}$ such that $F_{i,j}$ is $G_{v_i,v_j}$ for every
$1\leq i<j\leq k$.

While $\mathbb{G}(G,k)$ comes with labeled nodes, it is clear that
this graph with any other labeling of its nodes arises with the same
probability.

\subsection{Subgraph densities and moments}

From now on, $\KK$ denotes a compact separable topological space, and
we consider $\KK$-decorated graphs. Let $\CC$ denote the family of
continuous real valued functions on $\KK$, and let $\FF\subseteq\CC$.
For an $\FF$-decorated graph $F\in\GG_k(\FF)$ and a $\KK$-decorated
graph $G\in\GG(\KK)$, we introduce the weight $w(f)$ of a function
$f:~[k]\mapsto V(G)$ by
\[
w(f)=\prod_{1\leq i<j\leq k}F_{i,j}(G_{f(i),f(j)}).
\]
The {\it homomorphism number} $\hom(F,G)$ is
\[
\hom(F,G)=\sum_{f:[k]\mapsto V(G)}w(f).
\]
We also define the {\it homomorphism density} by
\[
t(F,G)=\frac{\hom(F,G)}{|V(G)|^k},
\]
which is the expected value $E(w(f))$ for a random map $f:~[k]\mapsto
V(G)$.

\subsection{Convergence}

Our goal is to study the following convergence notion.

\begin{definition}
An infinite sequence $G_1,G_2,\dots$ of $\KK$-decorated graphs is
called {\it convergent} if $|V(G_i)|\to\infty$, and for every $k$ the
sampling processes $\{\mathbb{G}(G_i,k)\}_{i=1}^\infty$ are weakly
convergent in distribution. This means that for every
$k\in\mathbb{N}$ and continuous function
$f:~\GG_k(\KK)\mapsto\mathbb{R}$ the limit
$\lim_{i\to\infty}E(f(\mathbb{G}(G_i,k)))$ exists.
\end{definition}

We are going to characterize convergence of a graph sequence in terms
of homomorphism numbers from $\CC$-decorated graphs. Since there are
generally too many such graphs, we will also show that we can
restrict ourselves to graphs decorated by elements from an
appropriate subset of $\CC$. We need the following definition.

\begin{definition}
We say that a set $\FF\subseteq\CC$ is {\it dense} if for every
$\epsilon>0$ and $f\in\CC$ there is an $g\in\FF$ such that
$|g(x)-f(x)|\leq\epsilon$ for every $x\in\KK$. We say that
$\FF\subseteq\CC$ is a {\it generating system} if the linear space
generated by the elements of $\FF$ is dense.
\end{definition}

We will prove the following theorem.

\begin{theorem}[Equivalence of convergence notions]\label{convequiv}
Let $(G_1,G_2,\dots)$ be a sequence of $\KK$-decorated graphs with
$|V(G_n)|\to\infty$, and let $\FF$ be a generating system. Then the
following are equivalent:

\smallskip

{\rm (i)} $(G_1,G_2,\dots)$ is convergent;

\smallskip

{\rm (ii)} For every $\CC$-decorated graph $F$, the numerical
sequence $t(F,G_n)$ is convergent;

\smallskip

{\rm (iii)} For every $\FF$-decorated graph $F$, the numerical
sequence $t(F,G_n)$ is convergent.
\end{theorem}

\subsection{Limit objects}\label{LIMITS}

Let $\mathcal{P}(\KK)$ denote the set of probability Borel measures
on the compact space $\KK$. The Riesz representation theorem implies
that the set $\mathcal{P}(\KK)$ is a compact topological space with
the weak topology. (Recall that the weak topology is the weakest
topology such that the function $\mu\mapsto\int_{\KK}f~d\mu$ is
continuous for every continuous function $f:~\KK\to\mathbb{R}$.)

\begin{definition}
We denote by $\WW(\KK)$ the set of two variable Borel measurable
functions $W:~[0,1]^2\mapsto \mathcal{P}(\KK)$ such that
$W(x,y)=W(y,x)$ for every $(x,y)\in[0,1]^2$. Elements of
$\mathcal{W}(\KK)$ will be called $\KK$-{\it graphons}.
\end{definition}

An important fact about $\KK$-graphons is that they can be described
by sequences of real-valued measurable functions in the following
way. Let $W$ be a $\KK$-graphon and let $f\in\CC$. Define
$W_f:~[0,1]^2\mapsto\mathbb{R}$ by $W_f(x,y)=\int_{\KK}f~dW(x,y)$.
The function $W_f$ is a bounded measurable function taking values
between $\min(f)$ and $\max(f)$. The Riesz representation theorem
implies that the measures $W(x,y)$ are reconstructible from the
sequence $\{W_f(x,y)\}_{f\in\FF}$ if $\FF$ is a generating system.

We will say that $(W_f:~f\in\FF)$ is the {\it $\FF$-moment
representation} of $W$. The name refers to the fact that for various
natural choices of $\KK$ and $\FF$, the numbers $t(F,G)$ ($F\in\FF$)
behave similarly to the moments of a single-variable function. This
analogy is explained and exploited in \cite{LSz7}.

Every $\KK$-decorated graph $G$ gives rise to a $\KK$-graphon $W_G$
as follows. Let $V(G)=[n]$. We split the unit interval into $n$
intervals $J_1,\dots,J_n$ of length $1/n$, and let $W_G(x,y)=G(i,j)$
for $x\in J_i, y\in J_j$ (here we identify the element $G(i,j)\in\KK$
with the distribution concentrated on $G(i,j)$).

For every $\KK$-graphon $W$ and $\CC$-decorated graph $F$ we
introduce the homomorphism density $t(F,W)$ by
\[
t(F,W):=\int\limits_{x_1,x_2,\dots,x_k\in[0,1]}\prod_{1\leq i<j\leq
k}W_{F_{i,j}}(x_i,x_j)~dx_1~dx_2\dots dx_k.
\]
It is easy to see that for every $\KK$-decorated graph $G$ and
$\CC$-decorated graph $F$,
\[
t(F,W_G)=t(F,G).
\]
Note that if $F$ is $\FF$-decorated for some $\FF\subseteq\CC$, then
$t(F,W)$ is expressed in terms of the $\FF$-moment representation of
$W$.

We will prove that the limit of a convergent sequence of
$\KK$-decorated graphs can be represented by a $\KK$-graphon:

\begin{theorem}\label{convgen}
Let $\FF$ be a countable generating set and let $(G_1,G_2,\dots)$ be
a convergent sequence of $\KK$-decorated graphs. Then there is a
$\KK$-graphon $W$ such that $t(F,G_n)\to t(F,W)$ for every
$\CC$-decorated graph $F$.
\end{theorem}

We in fact prove a more general theorem about convergence of
$\KK$-graphons.

\begin{theorem}\label{convgen1}
Let $\FF$ be a countable generating set and let $W_1,W_2,\dots$ be a
sequence of $\KK$-graphons such that $(t(F,W_n):~n=1,2,\dots)$ is a
convergent sequence for every $F\in\GG(\FF)$. Then there is a
$\KK$-graphon $W$ such that $t(F,W_n)\to t(F,W)$ for every
$F\in\GG(\CC)$.
\end{theorem}

\subsection{Examples}

\begin{example}[Simple graphs]\label{EXA:SGRAPH}
Let $\KK$ be the discrete space with two elements called ``edge'' and
``non edge'' or shortly $1$ and $0$. The set $\CC$ consists of all
maps $\{0,1\}\to\R$, i.e., of all pairs $(f(0),f(1))$ of real
numbers. A natural generating subset (in fact, a basis) in $\CC$
consists of the pairs $f_0=(1,1)$ and $f_1=(0,1)$. Sampling,
convergence, and homomorphism densities correspond to these notions
introduced for simple graphs.

Every probability distribution on $\KK$ can be represented by a
number between $0$ and $1$ which is the probability of the element
``edge''. So a $\KK$-graphon is described by a symmetric measurable
function $W:~[0,1]^2\mapsto [0,1]$. This has been the motivating
example worked out in \cite{LSz1,BCLSSV,BCLSV1}.

One may, however, take another basis in $\CC$, namely the pair
$g_0=(0,1)$ and $g_1=(1,0)$. Then again $\FF$-decorated graphs can be
thought of as simple graphs, and $\hom(F,G)$ counts the number of
maps that preserve both adjacency and non-adjacency.
\end{example}

\begin{example}[Multicolored graphs]\label{EXA:MGRAPH}
Let $\KK$ be a finite set of ``colors'' with the discrete topology.
Continuous functions on $\KK$ can be thought of as vectors in
$\R^\KK$. The standard basis $\FF$ in this space corresponds to
elements of $\KK$, and so $\FF$-decorated graphs are just the same as
$\KK$-decorated graphs. The moment $t(F,G)$ is the probability that a
random map $V(F)\to V(G)$ preserves edge colors.

Probability distributions on $\KK$ can be described by the
probabilities of its points. So a $\KK$-graphon is represented by $k$
measurable functions $w_i:~[0,1]^2\mapsto[0,1]$ with $\sum_i
w_i(x,y)=1$.
\end{example}

\begin{example}[Multigraphs]\label{EXA:MULTGRAPH}
Let $G$ be a multigraph with edge multiplicities at most $d$. Then
$G$ can be thought of as a $\KK$-decorated graph, where
$\KK=\{0,1,\dots,d\}$. This seems equivalent to example
\ref{EXA:MGRAPH}; however, it is more interesting to take a different
basis in $\R^\KK$ in this case, namely the functions
$\FF=\{1,x,\dots,x^d\}$. We can represent an $\FF$-decorated graph by
a multigraph with edge multiplicities at most $d$, where an edge
decorated by $x^i$ is represented by $i$ parallel edges. The
advantage of this is that $\hom(F,G)$ is then the number of
homomorphisms of $F$ into $G$ as multigraphs.
\end{example}

\begin{example}[Parallel colored graphs]\label{EXA:PMGRAPH}
This example is basically the same as Example \ref{EXA:MGRAPH} in the
sense that we choose $\KK$ to be finite; however we allow an edge to
carry more than one color. We take $\KK=\{0,1\}^n$, and let
$(e_1,e_2,\dots,e_n)\in\KK$ mean that an edge has color $i$ if and
only if $e_i=1$.

We have a fairly nice basis $\FF$: For every vector
$x=(x_1,x_2,\dots,x_n)\in\{0,1\}^n$ we construct a function
$f_x:~\KK\mapsto\{0,1\}$ such that $f_x(c_1,c_2,\dots,c_n)=1$ if and
only of $x_i=1$ implies $c_i=1$. These functions form a basis in
$\R^\KK$.

Limit objects are more complex. The most natural way is to use
$2^n-1$ measurable functions $w:~[0,1]^2\mapsto [0,1]$ whose sum is
between $0$ and $1$.

The main application of this example is that it allows us to study a
parallel limit of many graphs on the same node set.
\end{example}

\begin{example}[Infinitely many parallel graphs]\label{EXA:INFGRAPH}
The previous example can be further generalized by allowing
infinitely many parallel graphs. Let $\KK=\{0,1\}^{\mathbb{N}}$ be
the compact space with the product topology. Everything goes similar
to the previous example except that in the definition of $\FF$ we
only allow finitely many nonzero entries in the vector $x$ to
guarantee that the functions $f_x:\KK\mapsto\{0,1\}$ are continuous.
Limit objects now can be represented by an $\FF$-moment sequence of
measurable functions. Every such function is indexed by a finite
subset of the natural numbers.
\end{example}

\begin{example}[Weighted graphs]\label{EXA:WGRAPH}
Let $\KK\subseteq\R$ be a bounded closed interval. Let $\FF$ be the
collection of mononomial functions $x\mapsto x^j$ for $j=0,1,2,\dots$
on $\KK$; then $\FF$ is a generating system. It is natural to
consider an $\FF$-decorated graph $F$ as a multigraph, and then
$\hom(F,G)$ is the weighted homomorphism as defined e.g.\ in
\cite{FLS}.
\end{example}

\begin{example}[Compact topological groups]\label{EXA:CTGROUPS}
Let $\KK$ be a compact topological group. It is natural to choose
$\FF$ to be the Pontrjagin dual of $\KK$, which is the (discrete)
group of continuous homomorphisms from $\KK$ to $\C$. In the special
case $\KK=\mathbb{R}/\mathbb{Z}$, the dual group is isomorphic to the
integers, so every $\FF$-decorated graph can be considered as graphs
with multiple edges such that negative edge multiplicities are
allowed.
\end{example}

\section{Tools and proofs}

\subsection{Weak regularity partitions}\label{WEAKSZEM}

For a measurable function $W:~[0,1]^2\to \R$ we define its {\em
rectangle norm} by
\begin{equation}\label{SQUAREDEF}
\|W\|_\square=\sup_{{A\subseteq [0,1]}\atop{B\subseteq
[0,1]}}\left|\int_A\int_B W(x,y)\,dx\,dy\right|
\end{equation}
where $A$ and $B$ ranges over all possible measurable subsets of
$[0,1]$. It is easy to see that this norm could be defined by the
formula
\begin{equation}\label{SQUAREALT}
\|W\|_\square= \sup_{0\le f,g \le 1}\left|\int_0^1\int_0^1
W(x,y)f(x)g(y)\right|,
\end{equation}
where $f$ and $g$ are measurable functions. It is not hard to see
that it would not matter much to take the supremum over all functions
with bounded absolute value:
\begin{equation}\label{SQUAREALT-2}
\|W\|_\square\le \sup_{|f|,|g| \le 1}\left|\int_0^1\int_0^1
W(x,y)f(x)g(y)\right| \le 4 \|W\|_\square.
\end{equation}
We note that the supremum in the middle is the $L_\infty\to L_1$
operator norm of $W$ as a kernel operator.

The following ``weak'' version of Szemer\'edi's Regularity Lemma was
proved (in the context of matrices) by Frieze and Kannan \cite{FK}
(see \cite{LSz2} for this analytic formulation):

\begin{lemma}\label{wszem}
For every $\eps>0$ there is a constant $k=k(\eps)$ such that for
every symmetric measurable function $W\in\WW$ there is a partition
$\PP=\{P_1,P_2,\dots,P_k\}$ of $[0,1]$ into $k$ measurable subsets
and a stepfunction $W'$ which is constant on the sets $P_i\times P_j$
such that
\[
\|W-W'\|_{\square}\leq\eps\|W\|_\infty.
\]
\end{lemma}

It is easy to see that (at the cost of increasing $k(\eps,d)$) we can
impose additional conditions on the partition $\PP$ and the function
$W'$. We can assume that $\PP$ refines another partition and that the
partition sets have the same measure. We can also assume that
$W'=W_\PP$, where $W_\PP$ is obtained by taking the average of $W$ on
each set $P_i\times P_j$. By iterating this we get the following.

\begin{lemma}\label{WEAK-SZEM-W}
For every $\eps>0$ and natural numbers $t$ and $p$ there is an
integer $k(\eps,t,p)>0$ such that for every partition $\PP$ of
$[0,1]$ into $p$ equal sized measurable subsets and every family of
functions $W_i\in\WW$ $(i=1,2,\dots,t)$, there exists a partition
$\SS=\{S_1.\dots, S_k\}$ of $[0,1]$ into $k=k(\eps,t,p)$ measurable
sets such that

{\rm(1)} each $S_i$ has the same measure $1/k$;

{\rm(2)} $\SS$ is a refinement of the partition $P$;

{\rm(3)} $\|W_i-(W_i)_\SS\|_\square\le\eps \|W_i\|_\infty$ holds for
each $0\leq i\leq t$.
\end{lemma}

\subsection{Equivalence of convergence}

We need a few lemmas. For a subset $\mathcal{F}\subseteq\CC$ let
$\mathcal{F}_n$ denote the set of functions on $\KK^n$ whose elements
are of the form
\[
(c_1,c_2,\dots,c_n)\mapsto\prod_{i=1}^n f_i(c_i)
\]
with each $f_i$ in $\mathcal{F}$. It is clear that the elements of
$\mathcal{F}_n$ are continuous functions from $\KK^n$ to
$\mathbb{R}$.

\begin{lemma}\label{dgen1}
$\CC_n$ is a generating system on $\KK^n$.
\end{lemma}

\begin{proof}
First of all observe that the liner space generated by $\CC_n$ is an
algebra of continuous functions (containing the constant $1$
function) on $\KK^n$. Using Stone-Weierstrass Theorem it is enough to
show that $\CC_n$ is a separating set. Let $c=(c_1,c_2,\dots,c_n)$
and $d=(d_1,d_2,\dots,d_n)$ be two distinct elements on $\KK^n$ such
that $c_i\neq d_i$. Then there is a continuous function $f$ with
$f(c_i)\neq f(d_i)$ on $\KK$. It is clear that the function
$\hat{f}(x_1,x_2,\dots,x_n)=f(x_i)$ is in $\CC_n$ and it separates
$c$ form $d$.
\end{proof}

\begin{lemma}\label{dgen2}
If $\mathcal{F}$ is a generating system on $\KK$ then so is
$\mathcal{F}_n$ on $\KK^n$.
\end{lemma}

\begin{proof}
Let $f:\KK^n\mapsto\mathbb{R}$ be an arbitrary continuous function.
Then by Lemma \ref{dgen1} for every $\epsilon>0$ there is a finite
set of continuous functions $f_{i,j}$ on $\KK$ such that
\[
\bigl|\sum_{i=1}^k\prod_{j=1}^n
f_{i,j}(c_j)-f((c_1,c_2,\dots,c_n))\bigr|\leq\epsilon
\]
for every $(c_1,c_2,\dots,c_n)$ in $\KK^n$. Note that in this formula
we don't need linear coefficients since they can be ``merged'' into
the terms $f_{i,j}$. Using that $\mathcal{F}$ is a generating system
we have that for every $\epsilon_2>0$ there is finite system of
functions $\{g_i\}_{i=1}^r$ in $\mathcal{F}$ and real numbers
$\lambda_{i,j,k}$ such that
\[
|f_{i,j}-\sum_{k=1}^r\lambda_{i,j,k}g_k|\leq\epsilon_2
\]
everywhere on $\KK$ and for every $i,j$. Now replacing every
$f_{i,j}$ by its approximation with the functions $g_i$ we get an
approximation of $f$ with linear combinations of elements from
$\mathcal{F}_n$ with a precision arbitrary close to $\epsilon$ if we
let $\epsilon_2$ go to $0$. This completes the proof.
\end{proof}

Now we are ready to prove Theorem \ref{convequiv}. First we observe
that the distributional convergence of $\mathbb{G}(G_i,k)$ implies
that for any graph $F$ in $\tilde\GG_k(\mathcal{F})$ the sequence
$t(F,G_i)$ is convergent. This follows immediately from the
definition of $t(F,G_i)$ since $t(F,G_i)=E(L(\mathbb{G}(G_i,k)))$ for
some function $L$ in $\mathcal{F}_{{k}\choose{2}}$ whose components
are the edge weights in $F$.

The other direction follows from Lemma \ref{dgen2} since the
functions $L$ occurring in the formula above form a generating
system.

\subsection{Function sequences}

An indexed set $s=(s_f:~f\in \FF)$, where $s_f:~[0,1]^2\to\R$ is a
bounded symmetric measurable function for each $f$ will be called an
{\it $\FF$-indexed function sequence}. Let $\SS$ be the set of all
$\FF$-indexed function sequences. For every $\FF$-decorated graph $F$
and $s\in \SS$, we can define the ``homomorphism density''
\[
t(F,s)= \int_{x\in [0,1]^{V(F)}}~\prod_{1\leq i<j\leq k}
s_{F_{i,j}}(x_i,x_j)\,dx.
\]

The next lemma shows the relation between the $\|.\|_{\square}$-norm
and the homomorphism densities into a sequence of functions.

\begin{lemma}\label{HOMSQUARE}
Let $u=(u_f:~f\in\FF)$ and $w=(w_f:~f\in\FF)$ be two indexed sets of
functions in $\WW$, and let $d_f=\max(\|u_f\|_\infty,
\|w_f\|_\infty)$. Then for every $\FF$-decorated graph $F$,
\[
|t(F,u)-t(F,w)|\le 4\Bigl(\prod_{ij\in E(F)}
d_{F_{i,j}}\Bigr)\sum_{ij\in E(F)}
\|u_{F_{i,j}}-w_{F_{i,j}}\|_\square.
\]
\end{lemma}

\begin{proof}
Let $E(F)=\{i_1j_1,\dots,i_mj_m\}$. We have
\begin{align*}
t(F,U)&-t(F,W)\\&=\int_{[0,1]^n} \Bigl(\prod_{r=1}^{m}
W_{F_{i_r,j_r}}(x_{i_r},x_{j_r})-\prod_{r=1}^{m}
U_{F_{i_r,j_r}}(x_{i_r},x_{j_r})\Bigr)\,dx.
\end{align*}
We can write the telescoping sum
\[
\prod_{r=1}^{m} W_{F_{i_r,j_r}}(x_{i_r},x_{j_r})-\prod_{r=1}^{m}
U_{F_{i_r,j_r}}(x_{i_r},x_{j_r})=\sum_{t=1}^{m} X_t(x_1,\dots,x_n),
\]
where
\begin{align*}
X_t(x_1,\dots,x_n) = & \Bigl( \prod_{r=1}^{t-1}
W_{F_{i_r,j_r}}(x_{i_r},x_{j_r})\Bigr) \Bigl(\prod_{r=t+1}^{m}
U_{F_{i_r,j_r}}(x_{i_r},x_{j_r})\Bigr)\\
&\times(W_{F_{i_t,j_t}}(x_{i_t},x_{j_t})-
U_{F_{i_t,j_t}}(x_{i_t},x_{j_t})).
\end{align*}
To estimate the integral of a given $X_t$ term, let us integrate
first the variables $x_{i_t}$ and $x_{j_t}$; then by
(\ref{SQUAREALT-2}),
\[
\left|\int_0^1\int_0^1 X_t(x_1,\dots,x_n)\,dx_{i_t}\,dx_{j_t} \right|
\le 4
\Bigl(\prod_{r=1}^{m}d_{F_{i,j}}\Bigr)\|U_{F_{i_t,j_t}}-W_{F_{i_t,j_t}}\|_\square,
\]
which completes the proof.
\end{proof}

The $\FF$-indexed function sequences most important for us will be
the $\FF$-moment representations of $\KK$-graphons. An {\it
$\mathcal{F}$-moment sequence} is a family $(a_f:~f\in\FF)$ of the
form
\[
a_f=\int_{\KK} f~d\mu
\]
where $\mu$ is a Borel probability measure on $\KK$. An {\it
$\mathcal{F}$-moment function sequence} is a family $(w_f:~f\in\FF)$
of functions $w_f\in\WW$ such that $(w_f(x,y):~f\in\FF)$ is an
$\mathcal{F}$-moment sequence for all $x,y\in[0,1]$. Clearly the
$\FF$-moment representation of a $\KK$-graphon is an
$\mathcal{F}$-moment function sequence. The next Lemma shows that the
converse also holds.

\begin{lemma}\label{LEM:F-MOM}
For every $\mathcal{F}$-moment function sequence $w$ there is a
$\KK$-graphon $W$ with $\FF$-moment representation $w$.
\end{lemma}

\begin{proof}
Since $w(x,y)$ is an $\FF$-moment sequence for all $x,y\in[0.1]$,
there is a probability measure $W(x,y)$ such that $\int_\KK
f\,dW(x,y)=w_f(x,y)$ for all $f\in\FF$. To prove that $W$ is a
$\KK$-graphon, we have to check that it is measurable as a map
$[0,1]^2\to\PP(\KK)$. Since the Borel sets in $\PP(\KK)$ are
generated by the sets $\{\mu:~\int_\KK g\,d\mu\ge 0$ ($g\in\CC$), it
suffices to check that the sets
\[
A_g=\{(x,y):~\int_\KK g\,dW(x,y) \ge 0
\]
are measurable for all $g\in\CC$. Since $\FF$ is generating, we have
functions $g_n\in\CC$ such that $\|g_n-g\|_\infty\le 1/n$ and $g_n$
is in the linear hull of $\FF$. Clearly $A_g = \cup_n A_{g_n+1/n}$,
so it suffices to show that $A_{g_n+1/n}$ is measurable. Let
$g_n=\sum_{k=1}^N \alpha_k f_k$, where $f_k\in\FF$ and
$\alpha_k\in\R$. Then
\[
\int_\KK \Bigl(g_n+\frac1n\bigr)\,dW(x,y) = \frac1n + \sum_{k=1}^N
\alpha_k \int_\KK f_k\,dW(x,y) = \frac1n + \sum_{k=1}^N w_{f_k}(x,y)
\]
is a measurable function of $(x,y)$, which proves that $A_{g_n+1/n}$
is measurable.
\end{proof}

\subsection{Simultaneous convergence}

The following Lemma is our main tool for constructing limit objects.

\begin{lemma}\label{simcon}
Let $\FF\subseteq\CC$ be a countable generating system. Let $Q$ be a
compact convex subset of $\R^\FF$, and let $s_1,s_2,\dots\in\SS$ be a
sequence such that $s_n(x,y)\in Q$ for $n=1,2,\dots$. Assume that
there are reals $d_f>0$ $(f\in\FF)$ such that $|(s_n)_f(x,y)|\le d_f$
for all $n,x$ and $y$. Furthermore, assume that
$(t(F,s_n):~n=1,2,\dots)$ is a convergent sequence for every
$F\in\GG(\FF)$. Then there is a sequence $w\in\SS$ such that
$w(x,y)\in Q$ for all $x,y$ and $t(F,s_n)\to t(F,w)$ for every
$\FF$-decorated graph $F$.
\end{lemma}

\begin{proof}
Let $\FF=\{f_1,f_2,\dots\}$. For each $t\ge 1$, define $h(t)$
recursively by $h(1)=1$ and
\[
h(t)= k\Bigl(\frac1{t\max\{d_{f_i}:~1\le i\le t\}},t,h(t-1)\Bigr),
\]
where $k$ is the function in Lemma \ref{WEAK-SZEM-W}.

For each $t\geq 1$ we construct a partition $\PP_t$ of $[0,1]$ into
$h(t)$ intervals of equal length, and a subsequence $Q_t$ of the
natural numbers by recursion as follows. Let $\PP_1=\{[0,1]\}$ and
$Q_1=\N$. For each $n$, let $\PP_{t,n}$ be a partition refining
$\PP_{t-1}$ with $h(t)$ partition classes given by Lemma
\ref{WEAK-SZEM-W}, when applied to the sequence
$((s_n)_{f_1},\dots,(s_n)_{f_t})$ and $\eps=1/t$. We can apply, for
each $n$, a measure preserving transformation to $[0,1]$ so that
$\PP_{t,n}$ becomes a partition $\PP_t$ into intervals. Since
$\PP_{t,n}$ is a refinement of $\PP_{t-1}$, we can do this
transformation so that $\PP_{t-1}$ remains a partition into
intervals.

Let $v_{f,n,t}=((s_n)_f)_{\PP_t}$ and
$v_{n,t}=(v_{f,n,t}:~f\in\FF)\in\SS$. We have by Lemma
\ref{WEAK-SZEM-W}
\begin{equation}\label{EQ:SNP}
\| v_{f,n,t} - (s_n)_f\|_\square \le \frac1t \|(s_n)_f\|_\infty\le
\frac1t d_f.
\end{equation}

Let $Q_t$ be an infinite subsequence of $Q_{t-1}$ such that for every
$1\leq a,b \leq h(t)$ and $f\in \FF_t$, the functions $v_{f,n,t}$
converge to a function $u_{f,t}$ for $n\to\infty$, $n\in Q_t$. Such a
subsequence can be selected since the functions $v_{f,n,t}$ are
stepfunctions with a fixed partition and they are uniformly bounded.
Let $u_t=(u_{f,t}:~f\in\FF)\in\SS$.

We claim that the functions $u_{f,t}$ converge to some symmetric
measurable function $w_f$ almost everywhere if $t\to\infty$. This
follows from the properties that $(u_{f,t})_{\PP_t}=u_{f,t-1}$ and
$|u_{f,t}|\leq d_f$. Using the convergence theorem of bounded
martingales, one gets the convergence as in \cite{LSz1}.

Let $w=(w_f:~f\in\FF)$. It is clear that $w(x,y)\in Q$ for almost all
$x,y\in[0,1]$, and we may change the limit functions on a set of
measure $0$ so that this holds everywhere. For any $\FF$-decorated
graph $F$ with $n$ nodes, there is a real number $A_F$ such that
\begin{equation}\label{EQ:SNA}
t(F,s_n)\to A_F\qquad (n\to\infty).
\end{equation}
We also have, trivially,
\begin{equation}\label{eq123}
t(F, u_t)\to t(F,w) \qquad (t\to\infty),
\end{equation}
and
\begin{equation}\label{EQ:FUV}
t(F,v_{n,t})\to t(F,u_t)\qquad (n\to\infty).
\end{equation}
By Lemma \ref{HOMSQUARE} and inequality \eqref{EQ:SNP} we obtain that
\begin{align*}
|t(F,v_{n,t})&-t(F,s_n)|\\
&\leq 4\Bigl(\prod_{ij\in E(F)} d_{F_{i,j}}\Bigr)\sum_{ij\in E(F)}
\|v_{F_{i,j},n,t}-(s_n)_{F_{i,j}}\|_\square\le \frac1t d_{F_{i,j}}.
\end{align*}
Let $n\to\infty$ $(n\in Q_t)$, then the left hand side tends to
$|t(F,u_t)- A_F|$ by \eqref{EQ:FUV} and \eqref{EQ:SNA}. Letting
$t\to\infty$, \eqref{eq123} implies that $A_F= t(F,w)$ as claimed.
\end{proof}

Now we are ready to complete the proof of our main theorem.

\begin{proof*}{Theorem \ref{convgen1}}
Let $Q\subseteq\mathbb{R}^{\FF}$ denote the set of all
$\mathcal{F}$-moment sequences. Clearly, $Q$ is a compact, convex set
in $\R^\FF$.

Let $s_n\in\SS$ be the $\FF$-moment representation of $W_n$. Applying
Lemma \ref{simcon}, we get that there is a $w\in\SS$ such that
$w(x,y)$ is an $\FF$-moment sequence for all $x,y\in[0,1]$, and
$t(F,W_n)\to t(F,w)$ for all $\FF$-decorated graph $F$. The
$\FF$-moment function sequence $w$ defines a $\KK$-graphon $W$ by
Lemma \ref{LEM:F-MOM}, which proves the Theorem.
\end{proof*}

\end{document}